\documentclass[12pt]{article}
\usepackage[utf8]{inputenc}

\usepackage{amsfonts,amsmath,amssymb,amsthm}
\usepackage{authblk}

\usepackage{graphicx}

\title{Graph of $uv$-paths in $2$-connected graphs \thanks{Partially supported by Conacyt, M\'exico.}}
\author{Eduardo Rivera-Campo}
\affil{Departamento de Matem\'aticas\\Universidad Aut\'onoma Metropolitana - Iztapalapa\\ \text{erc@xanum.uam.mx}}
\date{}

\newtheorem{lemma}{Lemma}
\newtheorem{theorem}{Theorem}
\newtheorem{corollary}{Collorary}

\begin{document}

\maketitle

\begin{abstract}
    For a $2$-connected graph $G$ and vertices $u,v$ of $G$ we define an abstract graph $\mathcal{P}(G_{uv})$ whose vertices are the paths joining $u$ and $v$ in $G$, where paths $S$ and $T$ are adjacent if $T$ is obtained from $S$ by replacing a subpath $S_{xy}$ of $S$ with an internally disjoint subpath $T_{xy}$ of $T$. We prove that $\mathcal{P}(G_{uv})$ is always connected and give a necessary and a sufficient condition for connectedness in cases where the cycles formed by the replacing subpaths are restricted to a specific family of cycles of $G$. 
\end{abstract}

\section{Introduction}

 For any vertices $x,y$ of a path $L$, we denote by $L_{xy}$ the subpath of $L$ that joins $x$ and $y$. Let $G$ be a $2$-connected graph and $u$ and $v$ be vertices of $G$. The \emph{$uv$ path graph of $G$} is the graph $\mathcal{P}(G_{uv})$ whose vertices are the paths joining $u$ and $v$ in $G$, where two paths $S$ and $T$ are adjacent if $T$ is obtained from $S$ by replacing a subpath $S_{xy}$ of $S$ with an internally disjoint subpath $T_{xy}$ of $T$. The $uv$ path graph $\mathcal{P}(G_{uv})$ is closely related to the graph $G(P,f)$ of $f$-monotone paths on a polytope $P$  (see C. A. Athanasiadis \emph{et al} \cite{ALZ,AER}), whose vertices are the $f$-monotone paths on $P$ and where two paths  $S$ and $T$ are adjacent if there is a $2$-dimensional face $F$ of $P$ such that $T$ is obtained from $S$ by replacing an $f$-monotone subpath of $S$ contained in $F$ with the complementary $f$-monotone subpath of $T$ contained in $F$. 

In Section \ref{Prem} we show that the graphs $\mathcal{P}(G_{uv})$ are always connected as is the case for the graphs $G(P,f)$. 

If $S$ and $T$ are adjacent paths in a $uv$-path graph $\mathcal{P}(G_{uv})$, then $S\cup T$ is a subgraph of $G$ consisting of a unique cycle $\sigma$ joined to $u$ and $v$ by disjoint paths $P_{u}$ and $P_{v}$.  See Figure \ref{figmonocle}.

\begin{figure}[h!]
\centering
  \includegraphics[width= 4in]{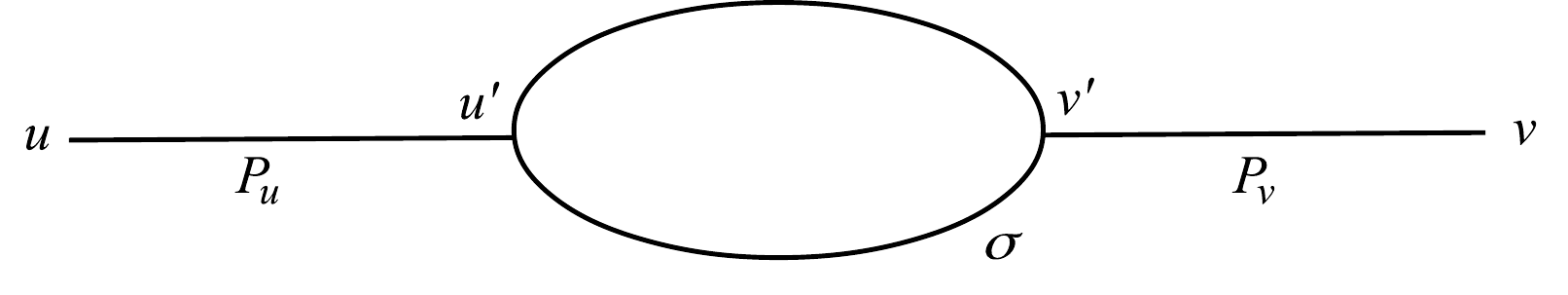}
  \caption{$S \cup T$}
  \label{figmonocle}
\end{figure}

Let $\mathcal{C}$ be a set of of cycles of $G$; the \emph{$uv$-path graph of $G$ defined by $\mathcal{C}$} is the spanning subgraph $\mathcal{P}_\mathcal{C}(G_{uv})$ of $\mathcal{P}(G_{uv})$ where two paths $S$ and $T$ are adjacent if and only if the unique cycle $\sigma$ which is contained in $S\cup T$ lies in $\mathcal{C}$. 
A graph $\mathcal{P}_\mathcal{C}(G_{uv})$ may be disconnected.

The $uv$ path graph $\mathcal{P}(G_{uv})$ is also related to the well-known \emph{tree graph} $\mathcal{T}(G)$ of a connected graph $G$, studied by R. L. Cummins \cite{C66}, in which the vertices are the spanning trees of $G$ and the edges correspond to pairs of trees $S$ and $R$ which are obtained from each other by a single edge exchange. As in the $uv$ path graph, if two trees $S$ and $R$ are adjacent in $\mathcal{T}(G)$, then $S \cup R$ is a subgraph of $G$ containing a unique cycle. X. Li \textit{et al}  \cite{LNR} define, in an analogous way, a subgraph $\mathcal{T}_\mathcal{C}(G)$ of $\mathcal{T}(G)$ for a set of cycles $\mathcal{C}$ of $G$ and give a necessary condition and a sufficient condition for $\mathcal{T}_\mathcal{C}(G)$ to be connected. In sections \ref{sectionnecessary} and \ref{sectionsufficient} we show that the same conditions apply to $uv$ path graphs $\mathcal{P}_\mathcal{C}(G_{uv})$.

Similar results are obtained by A. P. Figueroa \textit{et al} \cite{FFR} with respect to the \emph{perfect matching graph} $\mathcal{M}(G)$ of a graph $G$ where the vertices are the perfect matchings of $G$ and in which two matchings $L$ and $N$ are ajacent if their symmetric difference is a cycle of $G$. Again, if $L$ and $N$ are adjacent matchings in $\mathcal{M}(G)$, then $L\cup M$ contains a unique cycle of $G$. 

For any subgraphs $F$ and $H$ of a graph $G$, we denote by $F \Delta H$ the subgraph of $G$ induced by the set of edges $(E(F) \setminus E(H)) \cup (E(H) \setminus E(F))$.

\section{Preliminary results}
\label{Prem}

In this section we prove that the $uv$ path graph is  connected for any $2$-connected graph $G$ and give an upper bound for the diameter of a graph $\mathcal{P}(G_{uv})$.

\begin{theorem}
\label{connected}
Let $G$ be a $2$-connected graph. The $uv$-path graph $\mathcal{P}(G_{uv})$ is connected for every pair of vertices $u,v$ of $G$.  
\end{theorem}

\begin{proof}

For any different $uv$ paths $Q$ and $R$ in $G$ denote by $n(Q,R)$ the number of consecutive initial edges $Q$ and $R$ have in common. Assume the result is false and choose two $uv$ paths $S: u = x_0, x_1, \ldots, x_s = v$ and $T: u = y_0, y_1, \ldots, y_t = v$ in different components of $\mathcal{P}(G_{uv})$ for which $n^{*} = n(S,T)$ is maximum.

Since edges $x_{n^{*}}x_{n^{*}+1}$ and $y_{n^{*}}y_{n^{*}+1}$ are not equal,  $x_{n^{*}+1} \neq y_{n^{*}+1}$. Let $j = min \{i: x_{n^{*}+i} \in V(T)\}$ and $k = min \{i: y_{n^{*}+i} \in V(S)\}$ and let $l$ and $m$ be integers such that $y_{l} = x_{n^{*}+j} $, $x_m = y_{n^{*}+k}$. Consider the path:

$$S': u = x_0, x_1, \ldots, x_{n^*}, y_{n^* +1}, y_{n^* +2}, \ldots, y_{n^*+k}, x_{m+1},x_{m+2},  \ldots , x_s = v$$

Paths $S$ and $S'$ are adjacent in $\mathcal{P}(G_{uv})$ since $S'$ is obtained from $S$ by replacing the subpath $x_{n^{*}}, x_{n^{*}+1}, \ldots, x_{m}$ of $S$ with the subpath $y_{n^{*}},  y_{n^{*}+1}, \ldots, \allowbreak y_{n^{*}+k}$ of $S'$. Notice that $n(S', T) \geq n(S, T) + 1$ since $x_0x_1, x_1x_2, \ldots, x_{n^{*}-1}x_{n^{*}}, \allowbreak x_{n^{*}}y_{n^{*}+1} \in E(S') \cap E(T)$. By the choice of $S$, and $T$, paths $S'$ and $T$ are connected in $P(G_{uv})$. This implies that $S$ and $T$ are also connected in $\mathcal{P}(G_{uv})$ which is a contradiction.

\end{proof}
For any two vertices $u$ and $v$ of a connected graph $G$ we denote by $d_G(u,v)$ the distance between $u$ and $v$ in $G$, that is the length of a shortest $uv$ path in $G$. The diameter of a connected graph $G$ is the maximum distance among pairs of vertices of $G$. For a path $P$, we denote by $l(P)$ the length of $P$.

\begin{theorem}
\label{diameter}
Let $u$ and $v$ be vertices of a $2$-connected  graph $G$. The diameter of the graph  $\mathcal{P}(G_{uv})$ is at most $2d_G(u,v)$.
\end{theorem}

\begin{proof}
 Let $S$ and $T$ be $uv$ paths in $G$ and let $P$ be a shortest $uv$ path in $G$. From the proof of Theorem \ref{connected} one can see that there are two paths $Q_S$  and $Q_T$ in $\mathcal{P}(G_{uv})$, each with length at most $l(P)$, joining $S$ to $P$ and $T$ to $P$, respectively. Clearly $Q_S \cup Q_T$ contains a path  joining $S$ and $T$ in $\mathcal{P}(G_{uv})$ with length at most $2l(P)=2d_G(u,v)$. 
\end{proof}

In Figure \ref{figtight} we show a $2$-connected graph $G^2$ and paths $S$ and $T$ joining vertices $u$ and $v$ of $G^2$ such that $d_{G^2}(u,v) = 2$ and $d_{\mathcal{P}(G^2_{uv})}(S,T) = 4$. For any positive integer $k > 2$ the graph $G^2$ can be extended to a graph $G^k$ such that $d_{G^k}(u,v) = k$ and that the diameter of $\mathcal{P}(G^k_{uv}))$ is $2k$. This shows that Theorem \ref{diameter} is tight.

\begin{figure}[h!]
\centering
  \includegraphics[width= 4in]{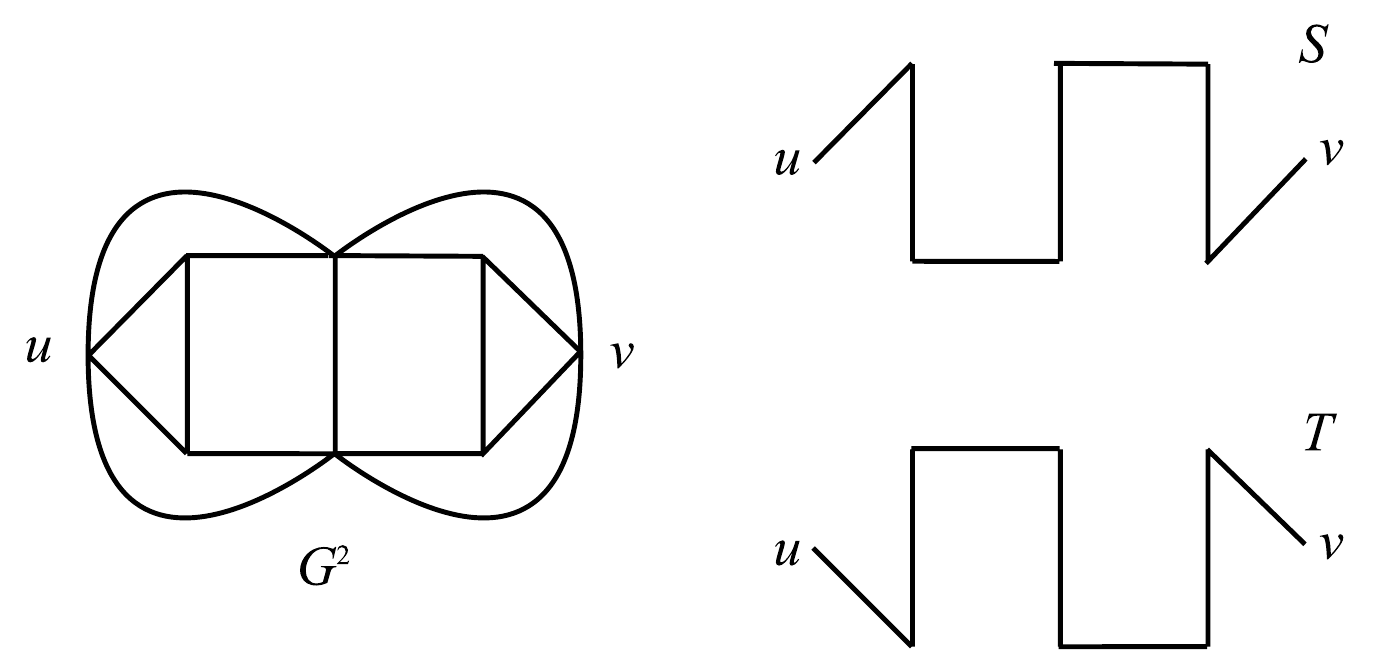}
  \caption{Graph $G^2$ and paths $S$ and $T$.}
  \label{figtight}
\end{figure}

\section{Necessary condition}
\label{sectionnecessary}

Let $u$ and $v$ be vertices of a $2$-connected graph $G$ and $S$ and $T$ be two $uv$ paths adjacent in $\mathcal{P}(G_{uv})$. Since $T$ is obtained from $S$ by replacing a subpath $S_{xy}$ of $S$ with an internally disjoint subpath $T_{xy}$ of $T$, the graph $S \Delta T$ is the cycle $S_{xy} \cup T_{xy}$.

\begin{theorem}
\label{necessary}
Let $G$ be a $2$-connected graph, $u$ and $v$ be vertices of $G$ and $\mathcal{C}$ be a set of cycles of $G$. If the graph $\mathcal{P}_\mathcal{C}(G_{uv})$ is connected, then $\mathcal{C}$ spans the cycle space of $G$.
\end{theorem}

\begin{proof}

Let $\sigma$ be a cycle of $G$. Since $G$ is $2$-connected, there are two disjoint paths $P_u$ and $P_v$ joining, respectively, $u$ and $v$ to $\sigma$. Denote by $u'$ and $v'$ the unique vertices of $P_u$ and $P_v$, respectively, that lie in $\sigma$. Vertices $u'$ and $v'$ partition cycle $\sigma$ into two internally disjoint paths $Q$ and $R$. Let $S = P_u \cup Q\cup P_v$ and $T = P_u \cup R \cup P_v$. Clearly $S$ and $T$ are two different $uv$ paths in $G$ such that $S \Delta T = \sigma$. 

Since $\mathcal{P}_\mathcal{C}(G_{uv})$ is connected, there are $uv$ paths $S = W_0, W_1, \ldots, W_k = T$ such that for $i=1, 2, \ldots k$, paths $W_{i-1}$ and $W_i$ are adjacent in $\mathcal{P}_\mathcal{C}(G_{uv})$. For $i=1, 2, \ldots k$ let $\alpha_i = W_{i-1} \Delta W_i$. Then $\alpha_1, \alpha_2, \ldots, \alpha_k$ are cycles in $\mathcal{C}$ such that:
$$\alpha_1 \Delta \alpha_2 \Delta \cdots \Delta \alpha_k = (W_0 \Delta W_1) \Delta (W_1 \Delta W_2) \Delta \cdots \Delta (W_{k-1} \Delta W_k ) = W_0 \Delta W_k = \sigma$$
Therefore $\mathcal{C}$ spans $\sigma$.

\end{proof}

Let $G$ be a complete graph with four vertices $u, x, y ,v$ and let $\mathcal{C} = \{\alpha, \beta, \delta \}$, where $\alpha = uxv$, $\beta = uyv$ and $\delta = uxyv$. Set $\mathcal{C}$ spans the cycle space of $G$ but the graph $\mathcal{P}_\mathcal{C}(G_{uv})$ is not connected since the $uv$ path $uyxv$ is an isolated vertex of $\mathcal{P}_\mathcal{C}(G_{uv})$, see Fig \ref{figK4}. This shows that the condition in Theorem \ref{necessary} is  not sufficient for $P_\mathcal{C}(G_{uv})$ to be connected.

\begin{figure}[h!]
\centering
  \includegraphics[width= 4in]{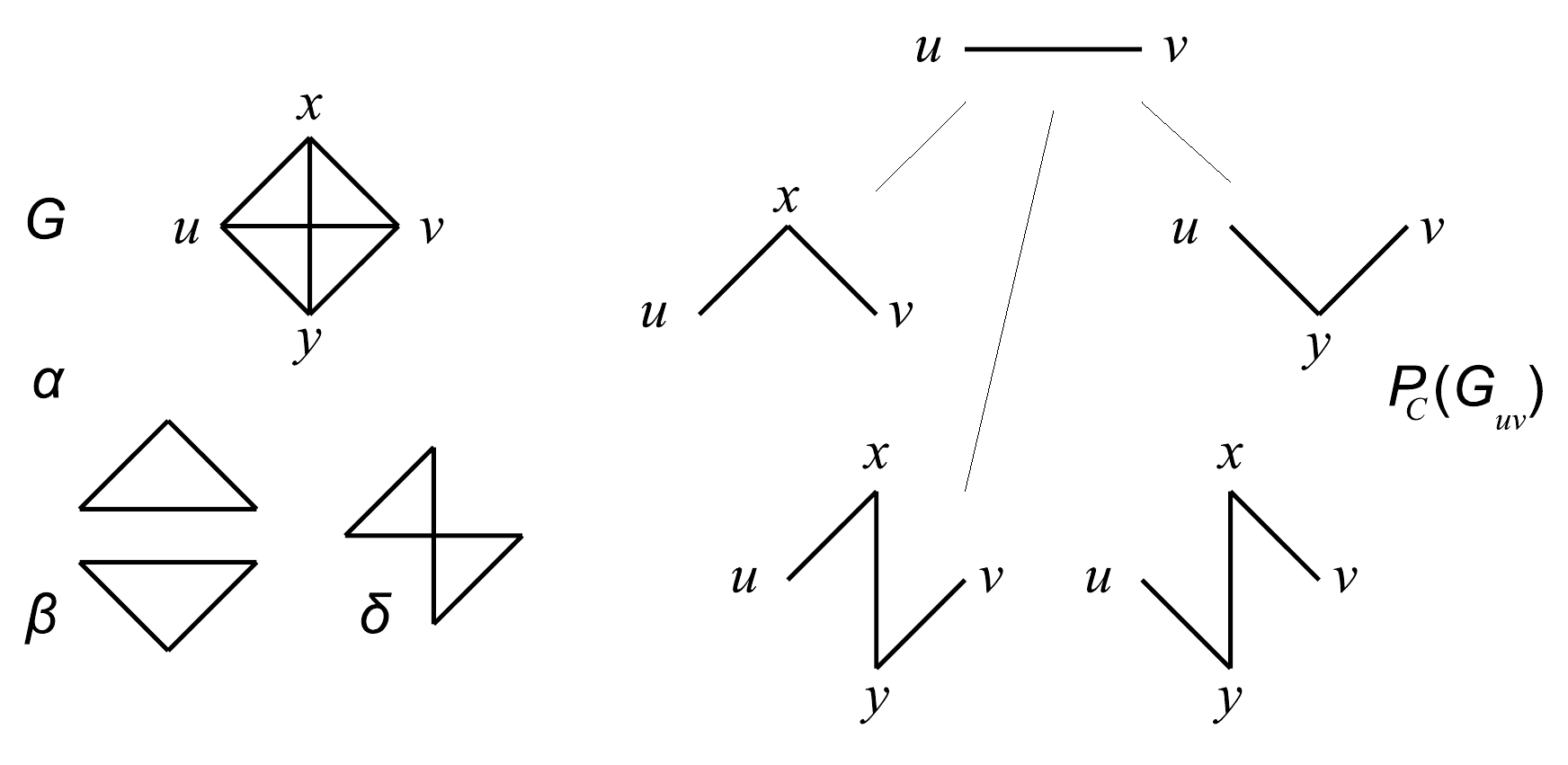}
  \caption{Graph $G$, set $\mathcal{C} = \{\alpha, \beta, \delta\}$ and graph $\mathcal{P}_C(G_{uv})$.}
  \label{figK4}
\end{figure}

\section{Sufficient condition}
\label{sectionsufficient}

A \emph{unicycle} of a connected graph $G$ is a spanning subgraph $\mathcal{U}$ of $G$ that contains a unique cycle. Let $u$ and $v$ be vertices of a $2$-connected graph $G$. A $uv$-\emph{monocle} of $G$ is a subgraph of $G$ that consists of a cycle $\sigma$ and two disjoint paths $P_u$ and $P_v$ that join, respectively $u$ and $v$ to $\sigma$, see Fig. \ref{figmonocle}. Clearly for each $uv$-monocle $\mathcal{M}$ of a $2$-connected graph $G$, there is a unicycle $\mathcal{U}$ of $G$ that contains $\mathcal{M}$.

Let $\mathcal{C}$ be a set of cycles of $G$. A cycle $\sigma$ of $G$ has \emph{Property $\Delta^*$} with respect to $\mathcal{C}$ if for every unicycle $\mathcal{U}$ containing $\sigma$ there is an edge $e$ of $G$, not in $\mathcal{U}$ and two cycles  $\alpha, \beta \in \mathcal{C}$, contained in $\mathcal{U} + e$, such that $\sigma = \alpha \Delta \beta$.

\begin{lemma}
\label{principal}
Let $G$ be a $2$-connected graph and $u$ and $v$ be vertices of $G$. Also let $\mathcal{C}$ be a set of cycles of $G$ and $\sigma$ be a cycle having Property $\Delta^*$  with respect to $\mathcal{C}$. The graph $\mathcal{P}_{\mathcal{C \cup \{\sigma\}}}(G_{uv})$ is connected if and only if $\mathcal{P}_{\mathcal{C}}(G_{uv})$ is connected. 
\end{lemma}

\begin{proof}
If $\mathcal{P}_{\mathcal{C}}(G_{uv})$ is connected, then $\mathcal{P}_{\mathcal{C \cup \{\sigma\}}}(G_{uv})$ is connected since the former is a subgraph of the latter. 

Assume now $\mathcal{P}_{\mathcal{C \cup \{\sigma\}}}(G_{uv})$ is connected and let $S$ and $T$ be $uv$ paths in $G$ which are adjacent in $\mathcal{P}_{\mathcal{C \cup \{\sigma\}}}(G_{uv})$. We show next that $S$ and $T$ are connected in $\mathcal{P}_\mathcal{C}(G_{uv})$ by a path of length at most 2. 

If $\omega = S \Delta T \in \mathcal{C}$, then $S$ and $T$ are adjacent in $\mathcal{P}_\mathcal{C}(G_{uv})$. For the case $\omega = \sigma$ denote by $\mathcal{M}$ the $uv$-monocle given by $S \cup T$. 

Let $\mathcal{U}$ be a unicycle of $G$ containing $\mathcal{M}$. Since $\sigma$ has Property $\Delta^*$  with respect to $\mathcal{C}$, there exists an edge $e=xy$ of $G$, not in $\mathcal{U}$, and two cycles $\alpha, \beta \in \mathcal{C}$ contained in $\mathcal{U} + e $ such that $\sigma = \alpha \Delta \beta$. 

Let $x'$ and $y'$ denote the vertices in $\mathcal{M}$ which are closest in $\mathcal{U}$ to $x$ and $y$, respectively. Then there exists a path $R_{x'y'}$ in $G$, with edges in $E(\mathcal{U}+e) \setminus E(\mathcal{M})$ joining $x'$ and $y'$ and such that cycles $\alpha$ and $\beta$ are contained in $\mathcal{M} \cup R_{x'y'}$. We analyze several cases according to the location of $x'$ and $y'$ in $\mathcal{M}$.

Denote by $P_u$ and $P_v$ the unique paths, contained in $\mathcal{M}$, that join $u$ and $v$ to $\sigma$ and by $u'$ and $v'$ the vertices where $P_u$ and $P_v$, respectively, meet $\sigma$.\\

\noindent \emph{Case} 1.- $x' \in V(P_u)$, $y' \in V(P_v)$. Without loss of generality we assume $\alpha = S_{x'y'} \cup R_{y'x'}$ and $\beta = T_{x'y'} \cup R_{y'x'}$, see Fig. \ref{figcase1}.

\begin{figure}[h!]
\centering
  \includegraphics[width= 4.3in]{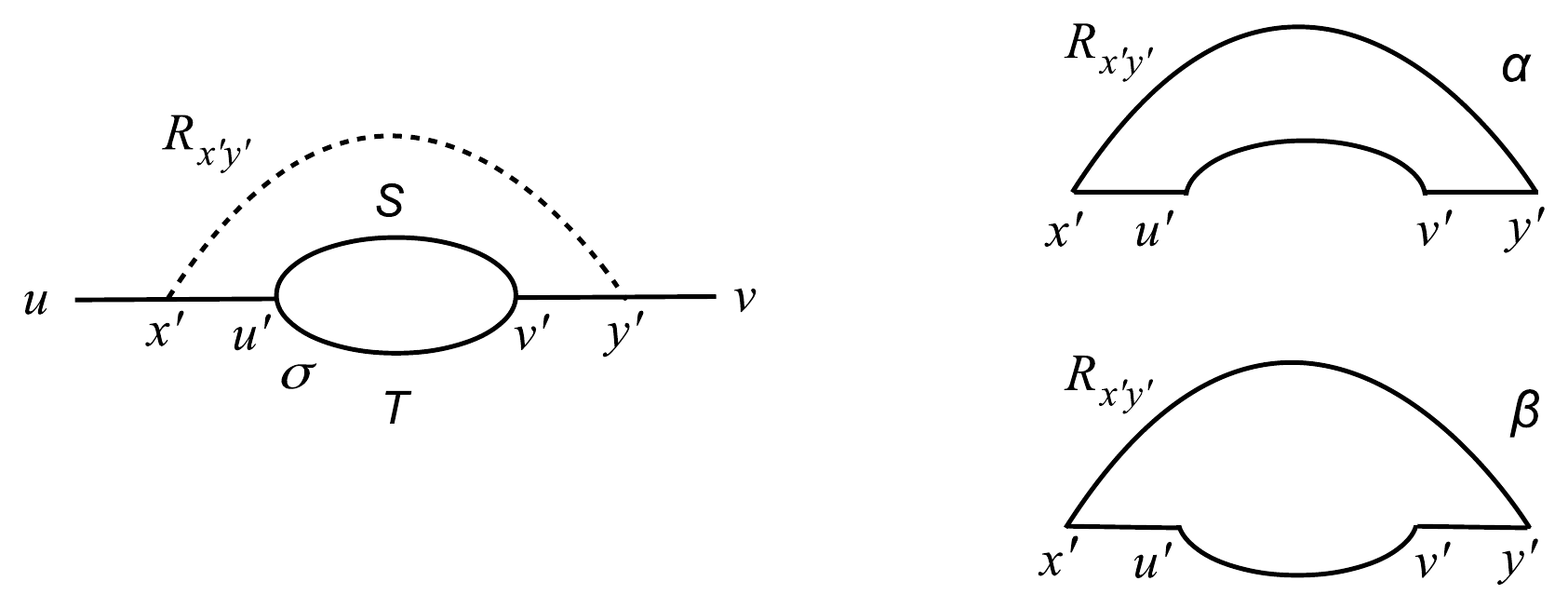}
  \caption{Left: $\mathcal{M} \cup R_{y'x'}$. Right: Cycles $\alpha$ and $\beta$.}
  \label{figcase1}
\end{figure}

Let $Q$ be the $uv$-path obtained from $S$ by replacing $S_{x'y'}$ with $R_{x'y'}$. Notice that $Q$ can also be obtained from $T$ by replacing $T_{x'y'}$ with $R_{x'y'}$.\\

\noindent \emph{Case} 2.- $x' \in V(P_u)$, $y' \in S \cap  \sigma$. Without loss of generality we assume $\alpha =  S_{x'y'} \cup R_{y'x'}$ and $\beta = T_{x'v'} \cup S_{v'y'} \cup R_{y'x'}$, see Fig. \ref{figcase2}.

\begin{figure}[h!]
\centering
  \includegraphics[width= 4.3in]{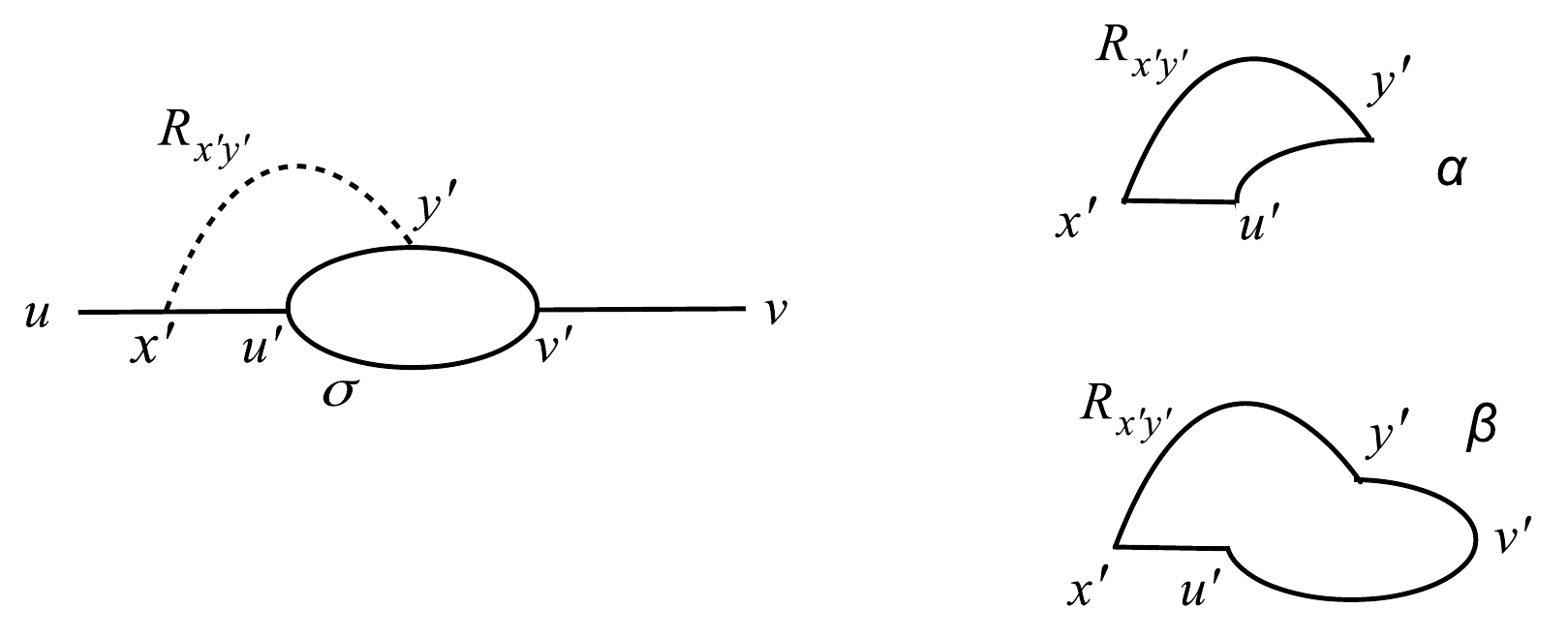}
 \caption{Left: $\mathcal{M} \cup R_{y'x'}$. Right: Cycles $\alpha$ and $\beta$.}
  \label{figcase2}
\end{figure}

Again let $Q$ be the $uv$-path obtained from $S$ by replacing $S_{x'y'}$ with $R_{x'y'}$. In this case, $Q$ can also be obtained from $T$ by replacing $T_{x'v'}$ with $R_{x'y'} \cup S_{y'v'}$.\\

\noindent \emph{Case} 3.- $x', y'  \in S \cap \sigma$. Without loss of generality we assume $\alpha = S_{x'y'} \cup R_{y'x'}$ and $\beta = S_{u'x'} \cup R_{x'y'} \cup S_{y'v'} \cup T_{v'u'}$ , see Fig. \ref{figcase3}.

\begin{figure}[h!]
\centering
  \includegraphics[width= 4.3in]{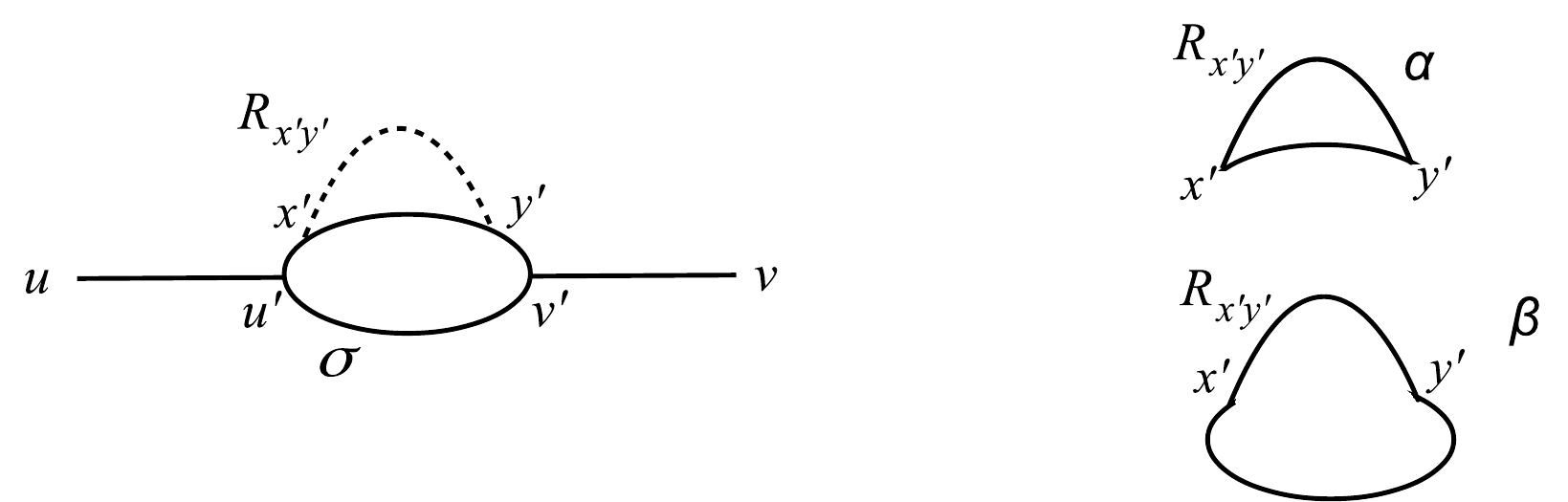}
\caption{Left: $\mathcal{M} \cup R_{y'x'}$. Right: Cycles $\alpha$ and $\beta$.}
  \label{figcase3}
\end{figure}

Let $Q$ be the $uv$-path obtained from $S$ by replacing $S_{x'y'}$ with $R_{x'y'}$. Path $Q$ is also obtained from $T$ by replacing $T_{u'v'}$ with $S_{u'x'} \cup R_{x'y'} \cup S_{y'v'}$.\\

\noindent \emph{Case} 4.- $x' \in S \cap \sigma$ and $y' \in T \cap \sigma$. Without loss of generality we assume $\alpha = S_{u'x'} \cup R_{x'y'} \cup T_{y'u'}$ and $\beta =  T_{y'v'} \cup S_{v'x'} \cup R_{x'y'} $. , see Fig. \ref{figcase4}.

\begin{figure}[h!]
\centering
  \includegraphics[width= 4.3in]{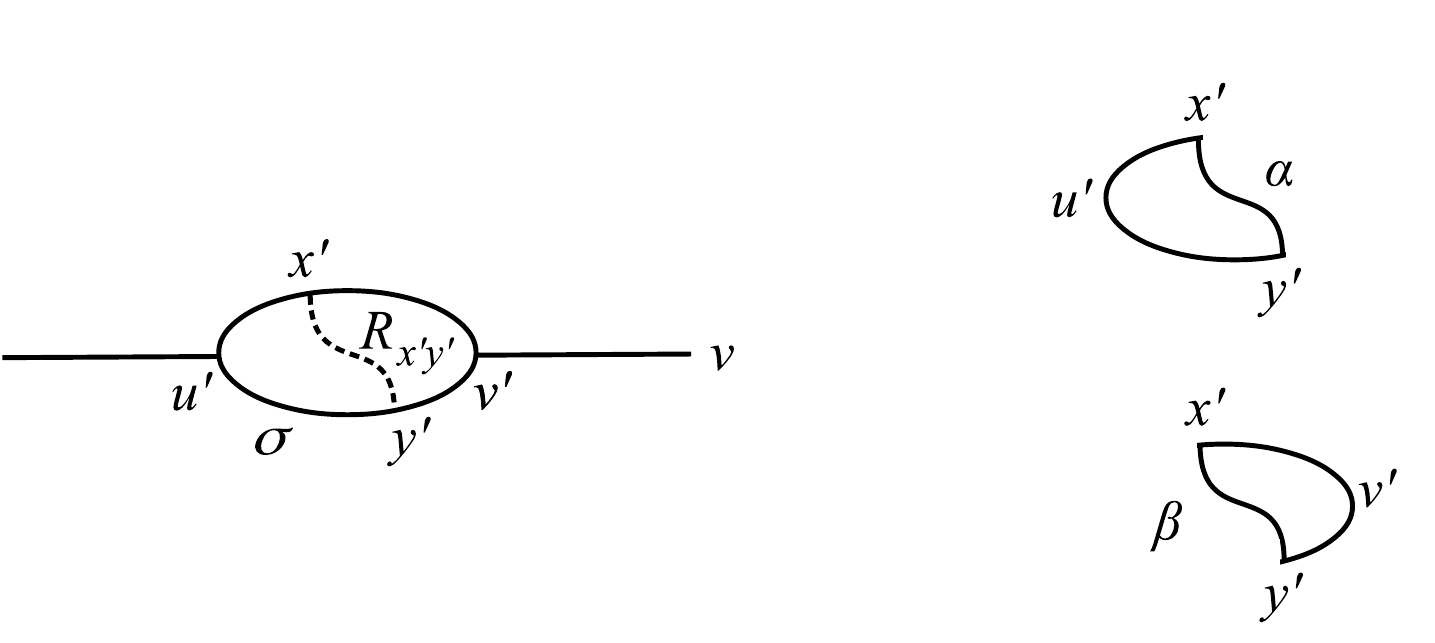}
 \caption{Left: $\mathcal{M} \cup R_{y'x'}$. Right: Cycles $\alpha$ and $\beta$.}
  \label{figcase4}
\end{figure}

Let $Q$ be the $uv$-path obtained from $S$ by replacing $S_{u'x'}$ with $T_{u'y'} \cup R_{y'x'}$. Now $Q$ can also be obtained from $T$ by replacing $T_{y'v'}$ with $R_{y',x'} \cup S_{x'v'}$\\

In each case $S \Delta Q = \alpha$ and $Q \Delta T = \beta$. Since $\alpha, \beta \in \mathcal{C}$, path $S$ is  adjacent to $Q$ and path $Q$ is adjacent to $T$ in $P_\mathcal{C}(G_{uv})$. Therefore $S$ and $T$ are connected in $\mathcal{P}_{\mathcal{C}}(G_{uv})$ by a path with length at most $2$. 

All remaining cases are analogous to either Case 2 or to Case 3. 
\end{proof}

Consider a $2$-connected graph $G$ with two specified vertices $u$ and $v$ and let $\mathcal{C}$ be a set of cycles of $G$. Construct a sequence of sets of cycles $\mathcal{C} = \mathcal{C}_0, \mathcal{C}_1, \ldots, \mathcal{C}_k$ as follows: If there is a cycle $\sigma_1$ not in $\mathcal{C}_0$  that has Property $\Delta^*$ with respect to $\mathcal{C}_0$ add $\sigma_1$ to $\mathcal{C}_0$ to obtain $\mathcal{C}_1$. At step $t$ add to $\mathcal{C}_t$ a new cycle $\sigma_{t+1}$ (if it exists) that has Property $\Delta
^*$ with respect to $\mathcal{C}_t$ to obtain $\mathcal{C}_{t+1}$. Stop at a step $k$ where there are no cycles, not in $\mathcal{C}_k$, having Property $\Delta^*$ with respect to $\mathcal{C}_k$. We denote by $Cl(\mathcal{C})$ the final set obtained with this process. Li \emph{et al} \cite{LNR} proved that the final set of cycles obtained is independent of which cycle $\sigma_t$ is added at each step in the case of multiple possibilities. 

A set of cycles of $G$ is \emph{$\Delta^*$-dense} if $Cl(\mathcal{C})$ is the whole set of cycles of $G$.

\begin{theorem}
\label{sufficient}
If $\mathcal{C}$ is $\Delta^*$-dense, then $\mathcal{P}_{\mathcal{C}}(G_{uv})$ is connected.
\end{theorem}

\begin{proof}

Since $\mathcal{C}$ is $\Delta^*$-dense, $Cl(\mathcal{C})$ is the set of cycles of $G$ and therefore $\mathcal{P}_{Cl(\mathcal{C})}(G_{uv}) = \mathcal{P}(G_{uv})$ which is connected by Theorem \ref{connected}.

Let $\mathcal{C}= \mathcal{C}_0 , \mathcal{C}_1, \ldots, \mathcal{C}_k = Cl(\mathcal{C})$ be a sequence of sets of cycles obtained from $\mathcal{C}$ as above. By Lemma \ref{principal}, all graphs $\mathcal{P}_{Cl(\mathcal{C})}(G_{uv}) = \mathcal{P}_{\mathcal{C}_{k}}(G_{uv}), \mathcal{P}_{\mathcal{C}_{k-1}}(G_{uv}), \allowbreak \ldots, \mathcal{P}_{\mathcal{C}_{0}}(G_{uv}) = \mathcal{P}_{\mathcal{C}}(G_{uv})$ are connected.
\end{proof}

Li \emph{et al} \cite{LNR} proved the following:

\begin{theorem}
\label{carasinternas}
If $G$ is a plane $2$-connected graph and $\mathcal{C}$ is the set of internal faces of $G$, then $\mathcal{C}$ is $\Delta^*$-dense. 
\end{theorem}

\begin{theorem}
\label{porunaarista}
If $G$ is a $2$-connected graph and $\mathcal{C}$ is the set of cycles that contain a given edge $e$ of $G$, then $\mathcal{C}$ is $\Delta^*$-dense. 
\end{theorem}
 
We end this section with the following immediate corollaries.

\begin{corollary}
\label{corcarasinternas}
Let $u$ and $v$ be vertices of a $2$-connected plane graph $G$. If $\mathcal{C}$ is the set of internal faces of $G$, then $\mathcal{P}_{\mathcal{C}}(G_{uv})$ is connected.
\end{corollary}

\begin{proof}
By Theorem \ref{carasinternas}, $\mathcal{C}$ is $\Delta^*$-dense and by Theorem \ref{sufficient}, $\mathcal{P}_{\mathcal{C}}(G_{uv})$ is connected.
\end{proof}

\begin{corollary}
\label{corporunaarista}
Let $u$ and $v$ be vertices of a $2$-connected graph $G$. If $\mathcal{C}$ is the set of cycles of $G$ that contain a given edge $e$, then  $\mathcal{P}_{\mathcal{C}}(G_{uv})$ is connected.
\end{corollary}

\begin{proof}
By Theorem \ref{porunaarista}, $\mathcal{C}$ is $\Delta^*$-dense and by Theorem \ref{sufficient}, $\mathcal{P}_{\mathcal{C}}(G_{uv})$ is connected.
\end{proof}

\begin{corollary}
\label{corporunvertice}
Let $u$ and $v$ be vertices of a $2$-connected graph $G$. If $\mathcal{C}_u$ is the set of cycles of $G$ that contain vertex $u$, then  $\mathcal{P}_{\mathcal{C}_u}(G_{uv})$ is connected.
\end{corollary}

\begin{proof}
Let $e$ be an edge of $G$ incident with vertex $u$. Clearly the set $\mathcal{C}(e)$ of cycles that contain edge $e$ is a subset of the set $\mathcal{C}_u$. Therefore $\mathcal{P}_{\mathcal{C}(e)}(G_{uv})$ is a subgraph of  $\mathcal{P}_{\mathcal{C}_u}(G_{uv})$. By Corollary \ref{corporunaarista}, the graph $\mathcal{P}_{\mathcal{C}(e)}(G_{uv})$ is connected. 
\end{proof}

\end{document}